\documentclass{article}
\usepackage{amsmath, amssymb, graphics}
\usepackage{amscd}
\usepackage{bm}

\advance \textwidth -60truemm\relax

\advance\oddsidemargin -1truein\relax

\evensidemargin=\oddsidemargin

\advance\textheight -70truemm\relax
\advance\topmargin -1truein\relax
\unitlength\textwidth
\divide\unitlength by 150\relax

\newcounter{thmno}
\numberwithin{thmno}{section}
\newtheorem{thm}[thmno]{Theorem}
\newtheorem{prop}[thmno]{Proposition}
\newtheorem{lem}[thmno]{Lemma}
\newtheorem{cor}[thmno]{Corollary}

\newtheorem{conj}[thmno]{Conjecture}

\bmdefine{\aaa}{a}
\bmdefine{\bbb}{b}
\bmdefine{\ccc}{c}
\bmdefine{\ddd}{d}
\bmdefine{\eee}{e}
\bmdefine{\uuu}{u}
\bmdefine{\xxx}{x}
\bmdefine{\zzz}{z}
\bmdefine{\zerovec}{0}
\def\FFF{\mathbb{F}}
\def\KKK{\mathbb{K}}
\def\CCC{\mathbb{C}}
\def\NNN{\mathbb{N}}
\def\RRR{\mathbb{R}}
\newcommand{\GL}{\mathrm{GL}}

\newcommand{\vect}{\mathrm{vec}}
\newcommand{\rank}{\mathrm{rank}}
\newcommand{\mrank}{\mathrm{mrank}}
\newcommand{\mtrank}{\mathrm{mtrank}}
\newcommand{\grank}{\mathrm{grank}}
\newcommand{\diag}{\mathrm{diag}}

\begin{document}

\title{Rank of tensors with size $2\times \cdots\times 2$}
\author{Toshio Sumi, Toshio Sakata and Mitsuhiro Miyazaki}
\maketitle

\begin{abstract}
We study an upper bound of ranks of $n$-tensors with size $2\times\cdots\times2$
over the complex and real number field.
We characterize a $2\times 2\times 2$ tensor with rank $3$ by using the
Cayley's hyperdeterminant and some function.
Then we see another proof of Brylinski's result that the maximal rank of $2\times2\times2\times2$ complex tensors is $4$.
We state supporting evidence of the claim that $5$ is a typical rank of $2\times2\times2\times2$ real tensors.
Recall that Kong and Jiang show that the maximal rank of $2\times2\times2\times2$ real tensors is less than or equal to $5$.
The maximal rank of $2\times2\times2\times2$ complex (resp. real) tensors gives an upper bound of the maximal rank of $2\times\cdots\times 2$ complex (resp. real) tensors.
\end{abstract}

\section{Introduction}
Let $\FFF$ be the real number field $\RRR$ or the complex number field $\CCC$.
For a positive integer $n$, an $n$-tensor $T$ over $\FFF$ with size $2\times \cdots\times 2$  is 
$$(t_{i_1,i_2,\ldots,i_n})$$
consisting of $2^n$ elements where $i_1,i_2,\ldots,i_n$ are taken $1$ and $2$, and 
$t_{i_1,i_2,\ldots,i_n}\in \FFF$ for $i_1,i_2,\ldots,i_n=1,2$.
Let $(\FFF^{2})^{\otimes n}$ be the set of all $n$-tensors over $\FFF$ with size $2\times \cdots\times 2$.
This set is closed by sum operation and scalar multiplication:
$$(t_{i_1,i_2,\ldots,i_n})+(s_{i_1,i_2,\ldots,i_n})=(t_{i_1,i_2,\ldots,i_n}+s_{i_1,i_2,\ldots,i_n})$$
$$c(t_{i_1,i_2,\ldots,i_n})=(ct_{i_1,i_2,\ldots,i_n})$$
And $\GL(2,\FFF)^n$ acts on the set $(\FFF^{2})^{\otimes n}$.
We call $T$ is a rank one tensor if $T$ is irreducible, that is, 
$T=T_1+T_2$ for some nonzero tensors $T_1,T_2$
implies that $T_1=sT_2$ for some $s\in \FFF$.
The rank of $T$, denoted by $\rank_{\FFF}(T)$ is the smallest integer $s\geq0$ such that
$T$ is expressed as the sum of $s$ rank one tensors.
The rank of the zero tensor is zero.  Rank is invariant under the $\GL(2,\FFF)^n$-action.
In general the determination of the rank of a tensor is hard.

The rank and classification of $2\times 2\times 2$ tensors are well-known, for example, 
see \cite{JaJa:1979b} or \cite{Sumi-etal:2009}.
The maximal rank of $2\times 2\times 2$ real tensors is equal to
one of $2\times 2\times 2$ complex tensors.

For a real $2\times 2\times 2\times 2$ tensor,
Kong et al. \cite{Kong-etal:2012} show that $\rank_{\RRR}(T)\leq 5$.

\begin{thm}[\cite{Kong-etal:2012}] \label{thm:Kong-etal}
Any real $2\times 2\times 2\times 2$ tensor has rank less than or equal to $5$.
\end{thm}

Brylinski gave the maximal rank of $2\times2\times2\times 2$ tensors over $\CCC$.

\begin{thm}[{\cite[Theorem~1.1]{Brylinski:2002}}] \label{thm:Brylinski}
Any complex $2\times2\times2\times 2$ tensor has rank less than or equal to $4$.
\end{thm}

$2\times 2\times 2\times 2$ tensors over $\FFF$
are used to represent the entanglement of four quantum bits (qubits). 
Verstraete et al \cite{Verstraete-etal:2002} gave a classification of $2\times 2\times 2\times 2$  rank one tensors.
We also show that $\rank_{\FFF}(A_1;A_2)\leq 2$ then $\rank_{\FFF}(T)\leq 4$ 
(see Propositions~\ref{prop:leq2implies4a}).
This was obtained by Kong et al. \cite{Kong-etal:2012} over $\RRR$.
By a numerical analysis, it seems that there are tensors over $\RRR$ with rank $5$.
There are tensors over the finite field $\mathbb{F}_3$ with rank $5$ 
(cf. \cite{Bremner-Stavrou:2012}).
The purpose of this paper is to give an upper bound of rank of tensors with size 
$2\times\cdots\times2$ by using the maximal rank of $2\times 2\times 2\times2$ real tensors (see Theorem~\ref{thm:higher2x...x2}).
Our main tool is a matrix theory.

\section{$2\times2\times 2$ tensors}

The maximal rank of $2\times2\times2$ tensors over $\FFF$ is equal to $3$.
In this section, we clarify a condition for a $2\times2\times2$ tensor to have rank three.

We denote the $2\times2$ identity matrix by $E_2$ or simply $E$.
Let $A=(a_{ij})=(\aaa_1,\aaa_2)$ and $B=(b_{ij})=(\bbb_1,\bbb_2)$ be $2\times 2$ matrices.
$\GL(2,\FFF)^3$ acts on the set of $2\times 2\times 2$ tensors by
$$(P,Q,R)\cdot (A;B)=(r_{11}PAQ^\top+r_{12}PBQ^\top;r_{21}PAQ^\top+r_{22}PBQ^\top),$$
where $R=\begin{pmatrix} r_{11}&r_{12}\\ r_{21}&r_{22}\end{pmatrix}$.
For a subgroup $G$ of $\GL(2,\FFF)^3$, two $2\times 2\times 2$ tensors $T_1$ and $T_2$ are $G$-equivalent if $g\cdot T_1=T_2$ for some $g\in G$.

\begin{prop} Suppose that $xA+yB$ are nonzero for any $(x,y)\ne(0,0)$ in $\FFF^2$.
If $(A;B)$ has rank two, then there is a tensor $X=(X_1;X_2)$ such that
$X$ is $\{E_2\}^2\times \GL(2,\FFF)$-equivalent to $(A;B)$ and $\det(X_1)=\det(X_2)=0$.
\end{prop}

\begin{proof}
Suppose that $(A;B)$ has rank two.  There are two rank one matrices $C_1$ and $C_2$
such that $A=pC_1+qC_2$ and $B=rC_1+sC_2$.  By the assumption, $P:=\begin{pmatrix} p&q\\ r&s\end{pmatrix}$ is nonsingular.  Let $X=(E_2,E_2,P^{-1})\cdot (A;B)$.
Then $X=(C_1;C_2)$.
\end{proof}

Remark that if $xA+yB$ are nonzero for any $(x,y)\ne(0,0)$, then $\rank(A;B)\geq 2$.

Define 
\begin{equation*}
\begin{split}
\Delta(A;B)&=(\det(A+B)-\det(A-B))^2/4-4\det(A)\det(B) \\
&=(\det(\aaa_1,\bbb_2)+\det(\bbb_1,\aaa_2))^2-4\det(\aaa_1,\aaa_2)\det(\bbb_1,\bbb_2).
\end{split}
\end{equation*}
This number is called Cayley's hyperdeterminant up to sign.
The discriminant of the polynomial $\det(xA+B)$ (resp. $\det(A+xB)$) on $x$ is equal to $\Delta(A;B)$ if $A$ (resp. $B$) is nonsingular.  
Thus, over $\RRR$, if $A$ (resp. $B$) is nonsingular and $\Delta(A;B)>0$ then there are $x_1$ and $x_2$ in $\RRR$
such that $x_1\ne x_2$ and $\det(x_1A+B)=\det(x_2A+B)=0$ (resp. $\det(A+x_1B)=\det(A+x_2B)=0$), and $\rank(A;B)\leq2$.
If $\det(A)=\det(B)=0$ then  $\rank(A;B)\leq \rank(A)+\rank(B)\leq 2$.
A $2\times 2\times 2$ tensor $T$ over a field with characteristic not $2$ 
is nonsingular if and only if $\Delta(T)$ is not a square in the field 
(see \cite{Coolsaet:2012}).

We show the following property straightforwardly. 
\begin{prop}[{\cite[Proposition 5.6]{Silva-etal:2008}}] \label{prop:Silva}
$$\Delta(A;B)=\Delta(B;A).$$
$$\Delta(A+xB;yB)=y^2\Delta(A;B)$$
for any $x$.
$$\Delta((P,Q,R)\cdot (A;B))=\Delta(A;B)\det(P)^2\det(Q)^2\det(R)^2$$
for any matrices $P$, $Q$ and $R$.
\end{prop}

If $A$ is nonsingular, $\rank(A;B)=2$ if and only if $A^{-1}B$ is diagonalizable.
If $A^{-1}B$ has distinct eigenvalues, the descriminant of the characteristic polynomial of $A^{-1}B$ which is equal to
\begin{equation*}
\begin{split}
(a_{11}b_{22}-a_{22}b_{11}+a_{12}b_{21}-a_{21}b_{12})^2+4(a_{22}b_{12}-a_{12}b_{22})(-a_{21}b_{11}+a_{11}b_{21}) \\
=(\det(\aaa_1,\bbb_2)-\det(\bbb_1,\aaa_2))^2-4\det(\aaa_1,\bbb_1)(\aaa_2,\bbb_2)
\end{split}
\end{equation*}
is positive in $\RRR$ and nonzero in $\CCC$.
Note that
$$\Delta(A;B)=(\det(\aaa_1,\bbb_2)-\det(\bbb_1,\aaa_2))^2-4\det(\aaa_1,\bbb_1)(\aaa_2,\bbb_2).$$
Thus we have the following proposition and theorem.

\begin{prop}[{\cite[Corollary 5.7, Propositions 5.9 and 5.10]{Silva-etal:2008}}] \label{prop:Silva-etal}
Let $T$ be a $2\times2\times 2$ real tensor.
\begin{enumerate}
\item The sign of $\Delta$ is invariant under the 
$\GL(2,\RRR)^{\times 3}$-action. 
\item If $\Delta(T)>0$ then $\rank(T)\leq 2$.
\item If $\Delta(T)<0$ then $\rank(T)=3$.
\item If $\rank(T)\leq 2$ then $\Delta(T)\geq 0$.
\end{enumerate}
\end{prop}

Put 
$S=\begin{pmatrix} 0&1\\ 0&0 \end{pmatrix}$ and
$R=\begin{pmatrix} 0&1\\ -1&0 \end{pmatrix}$.
Note that
$\Delta(E;S)=0$
and
$\Delta(E;R)=-4$.

\begin{thm} \label{thm:2x2x2}
Let $A=(\aaa_1,\aaa_2)$ and $B=(\bbb_1,\bbb_2)$ be $2\times2$ real (resp. complex) matices
and let $T=(A;B)$ be a $2\times 2\times 2$ tensor.
$\rank_{\FFF}(T)\leq 2$ if and only if at least one of the following holds: 
\begin{enumerate}
\item \label{thm:2x2x2:parallel1} 
$\alpha A+\beta B=O$ for
some $(\alpha,\beta)\ne(0,0)$,
\item \label{thm:2x2x2:parallel2} 
$\alpha (\aaa_1,\bbb_1)+\beta(\aaa_2,\bbb_2)=O$ for
some $(\alpha,\beta)\ne(0,0)$,
\item \label{thm:2x2x2:same} $\Delta(A;B)=0$ and $\det(\aaa_1,\bbb_1)+\det(\aaa_2,\bbb_2)=0$, or
\item \label{thm:2x2x2:different} $\Delta(A;B)$ is positive (resp. nonzero).
\end{enumerate}
\end{thm}

\begin{proof}
First suppose that $|x_1A+y_1B|\ne0$ for some $x_1$ and $y_1$.
There are $x_2$ and $y_2$ such that $\left|\begin{matrix} x_1&x_2\\ y_1&y_2\end{matrix}\right|\ne 0$ and $(A;B)$ is equivalent to $(x_1A+y_1B;x_2A+y_2B)$.
$\rank(A;B)\leq 2$ is equivalent to that $(x_1A+y_1B)^{-1}(x_2A+y_2B)$ is diagonalizable.
It is equivalent to
\begin{itemize}
\item[(i)]  $x_2A+y_2B=\alpha (x_1A+y_1B)$ for some $\alpha$, or
\item[(ii)]  all eigenvalues of $(x_1A+y_1B)^{-1}(x_2A+y_2B)$ lie in $\FFF$ and are distinct.
\end{itemize}
We have (i) $\Leftrightarrow$ \eqref{thm:2x2x2:parallel1}, and
(ii) $\Leftrightarrow$ $\Delta(x_1A+y_1B;x_2A+y_2B)$ is positive 
(resp. nonzero)
$\Leftrightarrow$ $\Delta(A;B)$ is positive (resp. nonzero).
\par
Next suppose that $|xA+yB|=0$ for any $x$ and $y$.
Then we have
$$\rank(A;B)\leq \rank(A)+\rank(B)\leq 1+1=2.$$
We see that $|xA+yB|=0$ for any $x$ and $y$ if and only if
$|A|=|B|=|\aaa_1,\bbb_2|+|\bbb_1,\aaa_2|=0$, since
$$|x\aaa_1+y\bbb_1,x\aaa_2+y\bbb_2|=x^2|\aaa_1,\aaa_2|+xy(|\aaa_1,\bbb_2|+|\bbb_1,\aaa_2|)+y^2|\bbb_1,\bbb_2|.$$
Thus, $|xA+yB|=0$ for any $x$ and $y$ if and only if $|A|=|B|=\Delta(A;B)=0$.
We divide into four cases: 
\begin{itemize}
\item[(a)] $\aaa_2=\alpha\aaa_1$, $\bbb_2=\beta\bbb_1$ for some $\alpha$ and $\beta$; 
\item[(b)] $\aaa_1\ne\zerovec$, $\aaa_2=\alpha\aaa_1$ for some $\alpha$, $\bbb_1=\zerovec$, $\bbb_2\ne\zerovec$; 
\item[(c)] $\aaa_1=\zerovec$, $\aaa_2\ne\zerovec$, $\bbb_1\ne\zerovec$, $\bbb_2=\beta\bbb_1$ for some $\beta$; and
\item[(d)] $\aaa_1=\zerovec$, $\aaa_2\ne\zerovec$, $\bbb_1=\zerovec$, $\bbb_2\ne\zerovec$.
\end{itemize}
(a) Since $|\aaa_1,\bbb_2|+|\bbb_1,\aaa_2|=(\beta-\alpha)|\aaa_1,\bbb_1|$,
if $|\aaa_1,\bbb_1|=0$ implies \eqref{thm:2x2x2:same} and otherwise $(\aaa_2,\bbb_2)=\alpha(\aaa_1,\bbb_1)$.
(b) Since $|\aaa_1,\bbb_2|+|\bbb_1,\aaa_2|=|\aaa_1,\bbb_2|$, we have
$|\aaa_1,\bbb_1|+|\aaa_2,\bbb_2|=\alpha|\aaa_1,\bbb_2|=0$ which implies \eqref{thm:2x2x2:same}.
(c) Since $|\aaa_1,\bbb_2|+|\bbb_1,\aaa_2|=|\bbb_1,\aaa_2|$, we have
$|\aaa_1,\bbb_1|+|\aaa_2,\bbb_2|=\beta|\bbb_1,\aaa_2|=0$ which implies \eqref{thm:2x2x2:same}.
(d) If $|\aaa_2,\bbb_2|=0$ implies \eqref{thm:2x2x2:same} and otherwise $(\aaa_1,\bbb_1)=O$ which implies \eqref{thm:2x2x2:parallel2}.
\end{proof}

We define a function $\Theta\colon \FFF^{2\times2\times2} \to \FFF$ by
$\Theta((\aaa_1,\aaa_2);(\bbb_1,\bbb_2))=|\aaa_1,\bbb_1|+|\aaa_2,\bbb_2|$.
We have the following corollary by Theorem~\ref{thm:2x2x2}.

\begin{cor} \label{cor:2x2x2}
Let $T=((\aaa_1,\aaa_2);(\bbb_1,\bbb_2))$ be a $2\times 2\times 2$ tensor.
\begin{enumerate}
\item \label{cor:2x2x2:complex}
A complex tensor $T$ has rank three if and only if
$\dim\langle \begin{pmatrix} \aaa_1\\ \aaa_2\end{pmatrix}, \begin{pmatrix} \bbb_1\\ \bbb_2\end{pmatrix}\rangle=\dim\langle \begin{pmatrix} \aaa_1\\ \bbb_1\end{pmatrix}, \begin{pmatrix} \aaa_2\\ \bbb_2\end{pmatrix}\rangle=2$, $\Delta(T)=0$ and $\Theta(T)\ne0$.
\item \label{cor:2x2x2:real}
A real tensor $T$ has rank three if and only if
$\Delta(T)<0$, or $\dim\langle \begin{pmatrix} \aaa_1\\ \aaa_2\end{pmatrix}, \begin{pmatrix} \bbb_1\\ \bbb_2\end{pmatrix}\rangle=\dim\langle \begin{pmatrix} \aaa_1\\ \bbb_1\end{pmatrix}, \begin{pmatrix} \aaa_2\\ \bbb_2\end{pmatrix}\rangle=2$, $\Delta(T)=0$ and $\Theta(T)\ne0$.
\end{enumerate}
\end{cor}

We put $A\cdot B=\det(A+B)-\det(A)-\det(B)$.

\begin{lem}[{\cite[Lemma~1]{Coolsaet:2012}}]
Let $\mathbb{K}$ be a field and $A, B \in \mathbb{K}^{2\times 2}$. 
Then $(A ; B)$ is nonsingular if and only if $\det(A)\ne 0$ and the
quadratic equation
$(\det A)x^2 - (A \cdot B)x + (\det B) = 0$,
has no solutions for $x$ in $\mathbb{K}$.
Equivalently, $(A ; B)$ is nonsingular if and only if $A$ is nonsingular and the eigenvalues
of $BA^{-1}$ do not belong to $\mathbb{K}$.
\end{lem}

\section{Theoretical results}

We refer to the paper \cite{Friedland:2008} by Friedland.
He wrote properties for $3$-tensors but
almost all properties with respect to the map $f_k$ canonically hold for $n$-tensors in general.

\begin{thm}[cf. {\cite[Theorem~7.1]{Friedland:2008}}]
\label{typical rank by Friedland}
The space $\mathcal{T}:=\RRR^{m_1\times\cdots\times m_n}$, $m_1,\ldots,,m_n \in \NNN$, contains a finite number of open connected
disjoint semi-algebraic sets $O_1, \ldots, O_M$ satisfying the following properties.
\begin{enumerate}
\item $\mathcal{T} \smallsetminus \cup_{i=1}^M O_i$ is a closed semi-algebraic set of $\mathcal{T}$ of dimension strictly less than $m_1\cdots m_n=\dim\mathcal{T}$.

\item Each $T \in O_i$ has rank $r_i$ for $i = 1,\ldots, M$.
\item $\min(r_1, \ldots, r_M) = grank(m_1,\ldots,m_n)$.
\item $mtrank(m_1,\ldots,m_n) := \max(r_1,\ldots, r_M)$ is the minimal $k \in \NNN$ such that the closure of $f_k((\RRR^{m_1}\times\cdots\times \RRR^{m_n})^k)$ is equal to $\mathcal{T}$.
\item For each integer $r \in [\grank(m_1,\ldots,m_n), \mtrank(m_1,\ldots,m_n)]$ there exists $r_i = r$ for some integer $i \in [1,M]$.
\end{enumerate}
\end{thm}

We call the number $r_i$, $i\in [1,M]$ a typical rank of $\mathcal{T}$.

We denote by $\mathcal{T}_{\leq p}$ the subset of all tensors with rank less than or equal to $p$ of $\mathcal{T}$.

\begin{thm} \label{thm:judge typical rank}
Let $\mathcal{T}:=\RRR^{m_1\times\cdots\times m_n}$, $p\in \NNN$ and $L$ a closed semi-algebraic set of dimension less than $\dim\mathcal{T}$.
Let $f\colon \mathcal{T}\smallsetminus L \to \RRR$ be a continuous map such that
$f(T)\geq 0$ for any $T$ and $f(T)=0$ for $T\in \mathcal{T}_{\leq p}$.
If $f$ is not zero, then there exists a typical rank $q$ of $\mathcal{T}$ with $q>p$.
\end{thm}

\begin{proof}
Suppose that there does not exist a typical rank $q$ of $\mathcal{T}$ with $q>p$.  Then $p$ is greater than or equal to the maximal typical rank of $\mathcal{T}$. By Theorem~\ref{typical rank by Friedland}, there is an open dense semi-algebraic set $O$ of $\mathcal{T}$ such that $O\subset \mathcal{T}_{\leq p}$
and then $O\smallsetminus L$ is also an open dense semi-algebraic set of $\mathcal{T}$.
Since $f(T)=0$ for any $T\in O\smallsetminus L$ and $f$ is continuous, we have
$f$ must be a constant zero map.
\end{proof}

\begin{cor}
Let $p$ be a typical rank of $\mathcal{T}$.
If there is a nonzero map in Theorem~\ref{thm:judge typical rank}, then
$p+1$ is a typical rank of $\mathcal{T}$.
\end{cor}

\section{$2\times 2\times 2\times 2$ tensors}

Let $\underline{T}=(t_{ijk\ell})$ be a $2\times 2\times2\times 2$ tensor. Put $T_{k\ell}=\begin{pmatrix} t_{11k\ell}& t_{12k\ell}\\
t_{21k\ell} & t_{22k\ell}\end{pmatrix}$ for $k,\ell=1,2$.
We describe $\underline{T}$ as 
$$\begin{array}{c|c}
T_{11} & T_{12}\vrule width0pt depth8pt\\
\hline
T_{21} & T_{22}\vrule width0pt height12pt
\end{array},\quad \begin{array}{c|c}
T_{\cdot1} & T_{\cdot2}
\end{array},\quad\text{or}\quad \begin{array}{c}
T_{1\cdot} \vrule width0pt depth8pt\\
\hline
T_{2\cdot} \vrule width0pt height12pt
\end{array}.$$
Let $G_1,\ldots,G_4=\GL(2,\CCC)$.
The action of $\GL(2,\CCC)^4$ is as follows.
$(P_1,P_2,E,E)\cdot \underline{T}$ is given by
$$\begin{array}{c|c}
P_1T_{11}P_2^\top & P_1T_{12}P_2^\top\vrule width0pt depth8pt\\
\hline
P_1T_{21}P_2^\top & P_1T_{22}P_2^\top\vrule width0pt height12pt
\end{array},$$
$(E,E,\begin{pmatrix} p_{11}&p_{12}\\ p_{21}&p_{22}\end{pmatrix},E)\cdot (T_{\cdot1}|T_{\cdot2})$ is given by
$$\begin{array}{c|c}
p_{11}T_{\cdot1}+p_{12}T_{\cdot2}&p_{21}T_{\cdot1}+p_{22}T_{\cdot2}\vrule width0pt depth8pt
\end{array},$$
and $(E,E,E,\begin{pmatrix} q_{11}&q_{12}\\ q_{21}&q_{22}\end{pmatrix})\cdot \left(\begin{array}{c} T_{1\cdot}\vrule width0pt depth5pt\\ \hline T_{2\cdot}\vrule width0pt height12pt\end{array}\right)$ is given by
$$\begin{array}{c} q_{11}T_{1\cdot}+q_{12}T_{2\cdot}\vrule width0pt depth8pt\\ \hline q_{21}T_{1\cdot}+q_{22}T_{2\cdot}\vrule width0pt height12pt\end{array}.$$
We also denote $\underline{T}$ by  $((T_{11};T_{12});(T_{21};T_{22}))$.

\begin{prop} \label{prop:action}
Let $g\in \GL(2,\FFF)^4$ and $\underline{T}^\prime=g\cdot \underline{T}$.
Suppose that $\rank_{\FFF}(T_{1\cdot})=3$.
If $\rank_{\FFF}(T^\prime_{1\cdot})\leq2$ then there is $x\in\FFF$ such that
$\rank_{\FFF}(xT_{1\cdot}+T_{2\cdot})\leq 2$.
\end{prop}

\begin{proof}
Suppose that $\rank_{\FFF}(T^\prime_{1\cdot})\leq2$ for $g=(g_1,g_2,g_3,g_4)$.
Then $\rank(T^{\prime\prime}_{1\cdot})=\rank_{\FFF}(T^\prime_{1\cdot})$, where $\underline{T}^{\prime\prime}=(E,E,E,g_4)\cdot \underline{T}$. 
Thus there is $(x_1,x_2)\ne(0,0)$ such that
$\rank(x_1T_{1\cdot}+x_2T_{2\cdot})=\rank(T^\prime_{1\cdot})$.
Note that $\rank(yS_{1\cdot})=\rank(S_{1\cdot})$ for any $y\ne 0$, where $\underline{S}=(g_1,g_2,g_3,E)\cdot \underline{T}$.
Since $\rank(T^{\prime\prime}_{1\cdot})\leq 2$, we have $x_2\ne 0$
and then $\rank(\frac{x_1}{x_2}T_{1\cdot}+T_{2\cdot})=\rank(T^\prime_{1\cdot})\leq2$.
\end{proof}

\begin{lem} \label{lem:onewayseparate}
Let $A$ and $B$ be real (resp. complex) $2\times 2\times 2$ tensors.
Suppose that $A=T_1+T_2$ for some rank one tensors $T_1$ and $T_2$.
If $\Delta(B+xT_1)$ is positive (resp. nonzero) for some $x$, then $\rank(A;B)\leq 4$.
\end{lem}

\begin{proof}
We replace the values of $\Delta$ are nonzero instead of positive over $\CCC$ in the following argument.
So, we assume that the base field is $\RRR$.
Suppose that $\Delta(B+xT_1)>0$.
Then $\rank(B+xT_1)=2$.  Therefore we have
\begin{equation*}
\begin{split}
\rank((A;B)+(-T_1;xT_1)) &=\rank(T_2;B+xT_1) \\
&\leq \rank(T_2)+\rank(B+xT_1) \\
&=1+2=3
\end{split}
\end{equation*}
and
\begin{equation*}
\begin{split}
\rank(A;B) &\leq \rank(T_1;-xT_1)+\rank(Y+(-T_1;xT_1)) \\
&\leq 1+3=4.
\end{split}
\end{equation*}
\end{proof}

\begin{lem} \label{lem;onepartiszero}
Let $\underline{Y}=((E;O);(B_1;B_2))$ be a $2\times 2\times 2\times 2$ tensor.
Then $\rank_{\FFF}(\underline{Y})\leq 4$ holds.
\end{lem}

\begin{proof}
Put
$$S_1=\diag(1,0), S_2=\diag(0,1), S_3=\frac{1}{2}\begin{pmatrix} 1&1\\ 1&1\end{pmatrix} 
\text{ and }
S_4=\frac{1}{2}\begin{pmatrix} 1&-1\\ -1&1\end{pmatrix}.$$
Note that $(E;O)=(S_1;O)+(S_2;O)=(S_3;O)+(S_4;O)$ and
$$\begin{array}{l}
\Delta((B_1;B_2)+x(S_1;O))=b_{222}^2x^2+(\text{lower term}), \\
\Delta((B_1;B_2)+x(S_2;O))=b_{112}^2x^2+(\text{lower term}), \\
\Delta((B_1;B_2)+x(S_3;O))=(b_{112} - b_{122} - b_{212} + b_{222})^2x^2/4+(\text{lower term}), \\
\Delta((B_1;B_2)+x(S_4;O))=(b_{112} + b_{122} - b_{212} - b_{222})^2x^2/4+(\text{lower term}).
\end{array}$$
If $b_{222}\ne 0$ then $\rank(B_1;B_2)+x_0(S_1;O))=2$ for some $x_0$
and $\rank(\underline{Y})\leq 4$ by Lemma~\ref{lem:onewayseparate}.
Similarly, 
if $b_{112}$, $b_{112} - b_{122} - b_{212} + b_{222}$ or $b_{112} + b_{122} - b_{212} - b_{222}$
is nonzero then $\rank(\underline{Y})\leq 4$.
If
$b_{222}=b_{112}=b_{112} - b_{122} - b_{212} + b_{222}=b_{112} + b_{122} - b_{212} - b_{222}=0$ 
then $B_2=0$ and thus $\rank(\underline{Y})\leq 3$, since $\rank(\underline{Y})=\rank(E;B_1)\leq 3$.
\end{proof}

\begin{prop}[cf. {\cite[Proposition~4.4]{Kong-etal:2012}}] \label{prop:leq2implies4a}
Let $\underline{Y}=((A_1;A_2);(B_1;B_2))$ be a $2\times 2\times 2\times 2$ tensor.
Suppose that $\rank_{\FFF}(A_1;A_2)\leq 2$.
Then $\rank_{\FFF}(\underline{Y})\leq 4$ holds.
\end{prop}

\begin{proof}
If $\rank(B_1;B_2)\leq 2$ then
we have
$$\rank(\underline{Y})\leq \rank(A_1;A_2)+\rank(B_1;B_2)\leq 2+2=4$$
and similarly if $\rank(A_1;A_2)\leq 1$ then
$$\rank(\underline{Y})\leq \rank(A_1;A_2)+\rank(B_1;B_2)\leq 1+3=4.$$
Suppose that $\rank(A_1;A_2)=2$ and $\rank(B_1;B_2)=3$.
Then there is $\mu\in \GL(2,\FFF)^{\times 3}$ such that
$\mu\cdot(A_1;A_2)=(E;\diag(a,b))$ for some $a,b\in\mathbb{R}$.
Putting $(C_1;C_2)=\mu\cdot(B_1;B_2)$,
we have
$$\Delta((C_1;C_2)+x(E;\diag(a,b)))=(a-b)^2x^4+(\text{lower term}).$$
Then if $a\ne b$ then $\Delta((B_1;B_2)+x_0(A_1;A_2))$ is positive (resp. nonzero) and thus
$\rank((B_1;B_2)+x_0(A_1;A_2))\leq 2$ for some $x_0$.
We have
\begin{equation*}
\begin{split}
\rank(\underline{Y})&\leq \rank(A_1;A_2)+\rank((B_1;B_2)+x_0(A_1;A_2)) \\
&\leq 2+2=4.
\end{split}
\end{equation*}
Finally, suppose that $a=b$.
Since $\underline{Y}$ is equivalent to $((E;O);(C_1;C_2-aC_1))$,
by Lemma~\ref{lem;onepartiszero}, we have $\rank(\underline{Y})\leq 4$.
\end{proof}

Over $\RRR$, the following proposition has been obtained 
(see \cite[Proposition 4.1]{Kong-etal:2012}).
For the reader's convenience, we show the proof of Theorem~\ref{thm:Kong-etal}.
\medskip

\begin{proof}
Let $\underline{Y}=(A;B)$ be a $2\times2\times2\times2$ real tensor.
If $\rank(A)\leq 2$, then $\rank(\underline{Y})\leq 4$ by Proposition~\ref{prop:leq2implies4a}.
Suppose that $\rank(A)>2$.  Then $\rank(A)=3$, since
the maximal rank of $2\times2\times2$ tensors is $3$. Take a $2\times2\times2$ rank one tensor $C$
such that $\rank(A-C)=2$.
Again by Proposition~\ref{prop:leq2implies4a}, we have $\rank(\underline{Y}-(C;O))\leq 4$.
Therefore, $\rank(\underline{Y})\leq \rank(\underline{Y}-(C;O))+\rank(C;O)\leq 4+1=5$.
This completes the proof.
\end{proof}

\begin{prop} \label{prop:exceptnone}
Let $\underline{T}$ be a $2\times 2\times2\times 2$ complex tensor.
There is a tensor $\underline{A}$ such that $\underline{A}$ is equivalent to $\underline{T}$ and 
$(A_{11};A_{12})$ has rank less than or equal to $2$.
\end{prop}

\begin{proof}
We may suppose that $\rank(T_{11};T_{12})=3$. 
There is $\alpha\in \GL(2,\FFF)^3$ such that $S_{11}=E$, $S_{12}=\begin{pmatrix} 0&1\\ 0&0\end{pmatrix}$
for $\underline{S}=(\alpha,E)\cdot \underline{T}$.
Since $\Theta(xS_{1\cdot}+S_{2\cdot})$ is a polynomial of $x$ with degree two,
there is $x_0\in\CCC$ such that $\Theta(x_0S_{1\cdot}+S_{2\cdot})=0$.
Thus by Corollary~\ref{cor:2x2x2} \eqref{cor:2x2x2:complex}, we have $\rank(x_0S_{1\cdot}+S_{2\cdot})\leq 2$.
Let $P=\begin{pmatrix} x_0&1\\ 1&0 \end{pmatrix}$ and $\underline{A}=(\alpha,P)\cdot\underline{T}$.
Then $\rank(A_{11};A_{12})\leq2$.
\end{proof}

\begin{thm}
The maximal and typical rank of $\CCC^{2\times2\times2\times2}$
is equal to $4$.
\end{thm}

\begin{proof}
By Propositions~\ref{prop:leq2implies4a} and \ref{prop:exceptnone},
the maximal rank of $\CCC^{2\times2\times2\times2}$
is equal to $4$.
For $\underline{A}=((A_{11};A_{12});(A_{21};A_{22}))\in \CCC^{2\times2\times2\times2}$, 
since 
$$\rank\underline{A}\geq 
\rank\begin{pmatrix} A_{11}&A_{12} \\ A_{21}&A_{22}\end{pmatrix},$$
a typical rank of $\CCC^{2\times2\times2\times 2}$ is greater than
or equal to $4$.
Therefore a typical rank of $\CCC^{2\times2\times2\times 2}$ is equal to $4$.
\end{proof}

Let
\begin{equation}\label{eq:exampletensor}
\underline{X}=\begin{array}{cc|cc}
1&0&0&1\\
0&1&-1&0\vrule width0pt depth4pt\\
\hline
0&-1&1&0\vrule width0pt height12pt\\
2&0&0&2\vrule width0pt depth4pt\\
\end{array}.
\end{equation}

We do not proceed in the real number field as in the complex number field:

\begin{prop} \label{prop:notConverttorank2slice}
There is a tensor $\underline{T}$ in $\RRR^{2\times2\times2\times2}$ such that
$\rank(S_{11};S_{12})=\rank(S_{11};S_{21})=\rank(S_{21};S_{22})=\rank(S_{12};S_{22})=3$
for any $g\in \GL(2,\RRR)^4$, where
$\underline{S}=\begin{array}{c|c} S_{11}&S_{12}\vrule width0pt depth4pt\\
\hline
S_{21}&S_{22}\vrule width0pt height12pt
\end{array}=g\cdot \underline{T}$.
\end{prop}

\begin{proof}
We show $\underline{T}=\underline{X}$ satisfies the assertion.
It suffices to show that $\Delta(S_{11};S_{12})$, $\Delta(S_{11};S_{21})$, $\Delta(S_{21};S_{22})$, $\Delta(S_{12};S_{22})$ are all negative for any $g\in \GL(2,\RRR)^4$.
Let $g\in \GL(2,\RRR)^4$.  By Proposition~\ref{prop:Silva}, we may suppose that $g=(E,E,P,Q)$.
For $\Delta(S_{11};S_{12})$, we may further suppose that $P=E$ by Proposition~\ref{prop:Silva}.
We straightforwardly see that  $\Delta(S_{11};S_{12})=-4(q_{11}^2 + 2q_{12}^2)^2<0$.
Similarly, for $\Delta(S_{11};S_{21})$, we may suppose that $Q=E$ and
see that $\Delta(S_{11};S_{21})=-8 (p_{11}^2 + p_{12}^2)^2<0$.
\end{proof}

Now we give a condition for a tensor $\underline{T}\in \RRR^{2\times2\times2\times2}$ to have rank $4$.

Let $\underline{T}=(t_{ijkl})$ be a $2\times 2\times2\times 2$ real tensor.
We consider the following condition (E):
\begin{eqnarray}
\left|\begin{matrix} 
t_{1111}&t_{1211}&t_{2111}&t_{2211} \\
c_{12}d_{12} & c_{12}d_{22} & c_{22}d_{12} & c_{22}d_{22} \\
c_{13}d_{13} & c_{13}d_{23} & c_{23}d_{13} & c_{23}d_{23} \\
c_{14}d_{14} & c_{14}d_{24} & c_{24}d_{14} & c_{24}d_{24} \end{matrix}\right|
\left|\begin{matrix} 
t_{1122}&t_{1222}&t_{2122}&t_{2222} \\
c_{12}d_{12} & c_{12}d_{22} & c_{22}d_{12} & c_{22}d_{22} \\
c_{13}d_{13} & c_{13}d_{23} & c_{23}d_{13} & c_{23}d_{23} \\
c_{14}d_{14} & c_{14}d_{24} & c_{24}d_{14} & c_{24}d_{24} \end{matrix}\right|
\hbox to 30mm{}\notag \\
-\left|\begin{matrix} 
t_{1112}&t_{1212}&t_{2112}&t_{2212} \\
c_{12}d_{12} & c_{12}d_{22} & c_{22}d_{12} & c_{22}d_{22} \\
c_{13}d_{13} & c_{13}d_{23} & c_{23}d_{13} & c_{23}d_{23} \\
c_{14}d_{14} & c_{14}d_{24} & c_{24}d_{14} & c_{24}d_{24} \end{matrix}\right|
\left|\begin{matrix} 
t_{1121}&t_{1221}&t_{2121}&t_{2221} \\
c_{12}d_{12} & c_{12}d_{22} & c_{22}d_{12} & c_{22}d_{22} \\
c_{13}d_{13} & c_{13}d_{23} & c_{23}d_{13} & c_{23}d_{23} \\
c_{14}d_{14} & c_{14}d_{24} & c_{24}d_{14} & c_{24}d_{24} \end{matrix}\right|=0,
\\
\left|\begin{matrix} 
c_{11}d_{11} & c_{11}d_{21} & c_{21}d_{11} & c_{21}d_{21} \\
t_{1111}&t_{1211}&t_{2111}&t_{2211} \\
c_{13}d_{13} & c_{13}d_{23} & c_{23}d_{13} & c_{23}d_{23} \\
c_{14}d_{14} & c_{14}d_{24} & c_{24}d_{14} & c_{24}d_{24} \end{matrix}\right|
\left|\begin{matrix} 
c_{11}d_{11} & c_{11}d_{21} & c_{21}d_{11} & c_{21}d_{21} \\
t_{1122}&t_{1222}&t_{2122}&t_{2222} \\
c_{13}d_{13} & c_{13}d_{23} & c_{23}d_{13} & c_{23}d_{23} \\
c_{14}d_{14} & c_{14}d_{24} & c_{24}d_{14} & c_{24}d_{24} \end{matrix}\right|
\hbox to 30mm{}\notag \\
-\left|\begin{matrix} 
c_{11}d_{11} & c_{11}d_{21} & c_{21}d_{11} & c_{21}d_{21} \\
t_{1112}&t_{1212}&t_{2112}&t_{2212} \\
c_{13}d_{13} & c_{13}d_{23} & c_{23}d_{13} & c_{23}d_{23} \\
c_{14}d_{14} & c_{14}d_{24} & c_{24}d_{14} & c_{24}d_{24} \end{matrix}\right|
\left|\begin{matrix} 
c_{11}d_{11} & c_{11}d_{21} & c_{21}d_{11} & c_{21}d_{21} \\
t_{1121}&t_{1221}&t_{2121}&t_{2221} \\
c_{13}d_{13} & c_{13}d_{23} & c_{23}d_{13} & c_{23}d_{23} \\
c_{14}d_{14} & c_{14}d_{24} & c_{24}d_{14} & c_{24}d_{24} \end{matrix}\right|=0,
\\
\left|\begin{matrix} 
c_{11}d_{11} & c_{11}d_{21} & c_{21}d_{11} & c_{21}d_{21} \\
c_{12}d_{12} & c_{12}d_{22} & c_{22}d_{12} & c_{22}d_{22} \\
t_{1111}&t_{1211}&t_{2111}&t_{2211} \\
c_{14}d_{14} & c_{14}d_{24} & c_{24}d_{14} & c_{24}d_{24} \end{matrix}\right|
\left|\begin{matrix} 
c_{11}d_{11} & c_{11}d_{21} & c_{21}d_{11} & c_{21}d_{21} \\
c_{12}d_{12} & c_{12}d_{22} & c_{22}d_{12} & c_{22}d_{22} \\
t_{1122}&t_{1222}&t_{2122}&t_{2222} \\
c_{14}d_{14} & c_{14}d_{24} & c_{24}d_{14} & c_{24}d_{24} \end{matrix}\right|
\hbox to 30mm{}\notag \\
-\left|\begin{matrix} 
c_{11}d_{11} & c_{11}d_{21} & c_{21}d_{11} & c_{21}d_{21} \\
c_{12}d_{12} & c_{12}d_{22} & c_{22}d_{12} & c_{22}d_{22} \\
t_{1112}&t_{1212}&t_{2112}&t_{2212} \\
c_{14}d_{14} & c_{14}d_{24} & c_{24}d_{14} & c_{24}d_{24} \end{matrix}\right|
\left|\begin{matrix} 
c_{11}d_{11} & c_{11}d_{21} & c_{21}d_{11} & c_{21}d_{21} \\
c_{12}d_{12} & c_{12}d_{22} & c_{22}d_{12} & c_{22}d_{22} \\
t_{1121}&t_{1221}&t_{2121}&t_{2221} \\
c_{14}d_{14} & c_{14}d_{24} & c_{24}d_{14} & c_{24}d_{24} \end{matrix}\right|=0,
\\
\left|\begin{matrix} 
c_{11}d_{11} & c_{11}d_{21} & c_{21}d_{11} & c_{21}d_{21} \\
c_{12}d_{12} & c_{12}d_{22} & c_{22}d_{12} & c_{22}d_{22} \\
c_{13}d_{13} & c_{13}d_{23} & c_{23}d_{13} & c_{23}d_{23} \\
t_{1111}&t_{1211}&t_{2111}&t_{2211} \end{matrix}\right|
\left|\begin{matrix} 
c_{11}d_{11} & c_{11}d_{21} & c_{21}d_{11} & c_{21}d_{21} \\
c_{12}d_{12} & c_{12}d_{22} & c_{22}d_{12} & c_{22}d_{22} \\
c_{13}d_{13} & c_{13}d_{23} & c_{23}d_{13} & c_{23}d_{23} \\
t_{1122}&t_{1222}&t_{2122}&t_{2222} \end{matrix}\right|
\hbox to 30mm{}\notag \\
-\left|\begin{matrix} 
c_{11}d_{11} & c_{11}d_{21} & c_{21}d_{11} & c_{21}d_{21} \\
c_{12}d_{12} & c_{12}d_{22} & c_{22}d_{12} & c_{22}d_{22} \\
c_{13}d_{13} & c_{13}d_{23} & c_{23}d_{13} & c_{23}d_{23} \\
t_{1112}&t_{1212}&t_{2112}&t_{2212} \end{matrix}\right|
\left|\begin{matrix} 
c_{11}d_{11} & c_{11}d_{21} & c_{21}d_{11} & c_{21}d_{21} \\
c_{12}d_{12} & c_{12}d_{22} & c_{22}d_{12} & c_{22}d_{22} \\
c_{13}d_{13} & c_{13}d_{23} & c_{23}d_{13} & c_{23}d_{23} \\
t_{1121}&t_{1221}&t_{2121}&t_{2221} \end{matrix}\right|=0,
\\
\det(M)\ne 0,
\end{eqnarray}
where 
\begin{equation} \label{eq:MatrixM}
M=\begin{pmatrix} 
c_{11}d_{11} & c_{11}d_{21} & c_{21}d_{11} & c_{21}d_{21} \\
c_{12}d_{12} & c_{12}d_{22} & c_{22}d_{12} & c_{22}d_{22} \\
c_{13}d_{13} & c_{13}d_{23} & c_{23}d_{13} & c_{23}d_{23} \\
c_{14}d_{14} & c_{14}d_{24} & c_{24}d_{14} & c_{24}d_{24} \end{pmatrix}
\end{equation}
is a $4\times 4$ matrix.

For a $2\times 2$ matrix $X=(\xxx_1,\xxx_2)$, put 
$\vect(X)=\begin{pmatrix} \xxx_1\\ \xxx_2\end{pmatrix}$.

\begin{thm} \label{thm:suffcond}
Suppose that $T_{11},T_{12},T_{21},T_{22}$ are linearly independent.
There are $c_{ik}$, $d_{ik}\in \KKK$ ($i=1,2$, $k=1,2,3,4$) such that the condition {\rm (E)} holds, if and only if $\rank_{\KKK}(\underline{T})=4$.
\end{thm}

\begin{proof}
Recall that
$\rank_{\KKK}(\underline{T})\leq r$ if and only if there are $x_{ik},y_{ik}\in \KKK$ and rank one matrices $A_k$
for $i=1,2$ and $k=1,2,\ldots,r$ such that $T_{ij}=\sum_{k=1}^r x_{ik}y_{jk}A_k$ for $i,j=1,2$.
Since $T_{11},T_{12},T_{21},T_{22}$ lie in the vector space generated by $A_k$'s,
we have $\rank_{\KKK}(\underline{T})\geq 4$.

Consider the problem $T_{ij}=\sum_{k=1}^4 c_{ik}d_{jk}A_k$ for $i,j=1,2$.
Then we have 
\begin{equation} \label{eq:B}
(\vect(T_{11}),\vect(T_{12}),\vect(T_{21}),\vect(T_{22}))
=(\vect(A_{1}),\vect(A_{2}),\vect(A_{3}),\vect(A_{4}))M.
\end{equation}
Since $(\vect(T_{11}),\vect(T_{12}),\vect(T_{21}),\vect(T_{22}))$ is nonsingular,
$M$ is nonsingular, and 
\begin{equation} \label{eq:A}
(\vect(A_{1}),\vect(A_{2}),\vect(A_{3}),\vect(A_{4}))=
(\vect(T_{11}),\vect(T_{12}),\vect(T_{21}),\vect(T_{22}))M^{-1}.
\end{equation}
Therefore, for each $k=1,2,3,4$, the equation ($k+1$) is zero if and only if $\rank(A_k)\leq1$.
Conversely, for $c_{ik}$, $d_{ik}\in \KKK$ such that the condition {\rm (E)} holds,
determining $A_k$ by \eqref{eq:A}, $A_k$ is a rank one matrix and 
$$\underline{T}=\sum_{k=1}^4 A_k\odot \begin{pmatrix} c_{1k}\\ c_{2k}\end{pmatrix}\odot
\begin{pmatrix} d_{1k}\\ d_{2k}\end{pmatrix},$$
which implies that $\rank_{\KKK}(\underline{T})=4$.
\end{proof}

\section{Numerical approach}

We compute $\Delta(A;B)$, $\Delta(C;D)$, $\Delta(A;C)$, $\Delta(B;D)$,
$\Delta(A+xC;B+xD)$, $\Delta(A+xB;C+xD)$, and
$\Delta(A+xB+z(C+xD);yA+B+z(yC+D))$ for $\begin{array}{c|c}
A& B\vrule width0pt depth4pt\\
\hline
C&D\vrule width0pt height12pt
\end{array}.$

There are $24$ flattening pattern corresponding to permutations $(i,j,k,\ell)$.
The following $4$ patterns are essentially same, since they correspond to the transpose of
matrices.
 
$$(t_{ijk\ell})=\begin{array}{cc|cc}
t_{1111}&t_{1211} &t_{1112} &t_{1212}\\
t_{2111}&t_{2211}&t_{2112}&t_{2212}\vrule width0pt depth4pt\\
\hline
t_{1121}&t_{1221} &t_{1122} &t_{1222}\vrule width0pt height12pt\\
t_{2121}&t_{2221}&t_{2122}&t_{2222}\vrule width0pt depth4pt\\
\end{array}, \quad
(t_{jik\ell})=\begin{array}{cc|cc}
t_{1111}&t_{2111} &t_{1112} &t_{2112}\\
t_{1211}&t_{2211}&t_{1212}&t_{2212}\vrule width0pt depth4pt\\
\hline
t_{1121}&t_{2121} &t_{1122} &t_{2122}\vrule width0pt height12pt\\
t_{1221}&t_{2221}&t_{1222}&t_{2222}\vrule width0pt depth4pt\\
\end{array},$$

$$(t_{ij\ell k})=\begin{array}{cc|cc}
t_{1111}&t_{1211} &t_{1121} &t_{1221}\\
t_{2111}&t_{2211}&t_{2121}&t_{2221}\vrule width0pt depth4pt\\
\hline
t_{1112}&t_{1212} &t_{1122} &t_{1222}\vrule width0pt height12pt\\
t_{2112}&t_{2212}&t_{2122}&t_{2222}\vrule width0pt depth4pt\\
\end{array},\quad
(t_{ji\ell k})=\begin{array}{cc|cc}
t_{1111}&t_{2111} &t_{1121} &t_{2121}\\
t_{1211}&t_{2211}&t_{1221}&t_{2221}\vrule width0pt depth4pt\\
\hline
t_{1112}&t_{2112} &t_{1122} &t_{2122}\vrule width0pt height12pt\\
t_{1212}&t_{2212}&t_{1222}&t_{2222}\vrule width0pt depth4pt\\
\end{array}$$

Thus there are essentially $6$ patterns:

$$(t_{ijk\ell})=\begin{array}{cc|cc}
t_{1111}&t_{1211} &t_{1112} &t_{1212}\\
t_{2111}&t_{2211}&t_{2112}&t_{2212}\vrule width0pt depth4pt\\
\hline
t_{1121}&t_{1221} &t_{1122} &t_{1222}\vrule width0pt height12pt\\
t_{2121}&t_{2221}&t_{2122}&t_{2222}\vrule width0pt depth4pt\\
\end{array},\quad 
(t_{kji\ell})=\begin{array}{cc|cc}
t_{1111}&t_{1211} &t_{1112} &t_{1212}\\
t_{1121}&t_{1221}&t_{1122}&t_{1222}\vrule width0pt depth4pt\\
\hline
t_{2111}&t_{2211} &t_{2112} &t_{2212}\vrule width0pt height12pt\\
t_{2121}&t_{2221}&t_{2122}&t_{2222}\vrule width0pt depth4pt\\
\end{array},$$

$$(t_{ikj\ell})=\begin{array}{cc|cc}
t_{1111}&t_{1121} &t_{1112} &t_{1122}\\
t_{2111}&t_{2121}&t_{2112}&t_{2122}\vrule width0pt depth4pt\\
\hline
t_{1211}&t_{1221} &t_{1212} &t_{1222}\vrule width0pt height12pt\\
t_{2211}&t_{2221}&t_{2212}&t_{2222}\vrule width0pt depth4pt\\
\end{array}, \quad
(t_{i\ell kj})=\begin{array}{cc|cc}
t_{1111}&t_{1112} &t_{1211} &t_{1212}\\
t_{2111}&t_{2112}&t_{2211}&t_{2212}\vrule width0pt depth4pt\\
\hline
t_{1121}&t_{1122} &t_{1221} &t_{1222}\vrule width0pt height12pt\\
t_{2121}&t_{2122}&t_{2221}&t_{2222}\vrule width0pt depth4pt
\end{array},$$

$$(t_{k\ell ij})=\begin{array}{cc|cc}
t_{1111}&t_{1112} &t_{1211} &t_{1212}\\
t_{1121}&t_{1122}&t_{1221}&t_{1222}\vrule width0pt depth4pt\\
\hline
t_{2111}&t_{2112} &t_{2211} &t_{2212}\vrule width0pt height12pt\\
t_{2121}&t_{2122}&t_{2221}&t_{2222}\vrule width0pt depth4pt\\
\end{array}, \quad
(t_{j\ell ki})=\begin{array}{cc|cc}
t_{1111}&t_{2111} &t_{1211} &t_{2211}\\
t_{1112}&t_{2112}&t_{1212}&t_{2212}\vrule width0pt depth4pt\\
\hline
t_{1121}&t_{2121} &t_{1221} &t_{2221}\vrule width0pt height12pt\\
t_{1122}&t_{2122}&t_{1222}&t_{2222}\vrule width0pt depth4pt
\end{array}$$

The tensor $\underline{X}$ given in \eqref{eq:exampletensor} implies the following result.

\begin{prop} \label{prop:notConverttorank2sliceforanyflatten}
There is a tensor $\underline{T}$ in $\RRR^{2\times2\times2\times2}$ such that
$\rank(S_{11};S_{12})=\rank(S_{11};S_{21})=\rank(S_{21};S_{22})=\rank(S_{12};S_{22})=3$
for any flattening pattern $\underline{T}^\prime$ of $\underline{T}$ and any $g\in \GL(2,\RRR)^4$, where
$\underline{S}=\begin{array}{c|c} S_{11}&S_{12}\vrule width0pt depth4pt\\
\hline
S_{21}&S_{22}\vrule width0pt height12pt
\end{array}=g\cdot \underline{T}^\prime$.
\end{prop}

Therefore, there is a tensor in $\RRR^{2\times2\times2\times2}$
which does not apply Proposition~\ref{prop:leq2implies4a}.

\begin{thm} \label{thm:estimate}
Let $\underline{T}$ be a $2\times 2\times2\times2$ real tensor.
Suppose that the $4\times4$ matrix obtained from $\underline{T}$ is nonsingular.
Put 
\begin{equation}\label{eq:function}
f((c_{1k},d_{1k},c_{2k},d_{2k})^\top_{1\leq k\leq 4})=\frac{\displaystyle\sum_{j=1}^4 (n_{1j}n_{4j}-n_{2j}n_{3j})^2}{\displaystyle(\sum_{i,j=1}^4 n_{ij}^2)^2}.
\end{equation}
where 
$$(n_{ij})=(\vect(T_{11}),\vect(T_{12}),\vect(T_{21}),\vect(T_{22}))M^{-1}.$$
If $\rank_{\RRR}(\underline{T})=4$ then $f=0$ and $\det(M)\ne0$ at some $(c_{1k},d_{1k}c_{2k},d_{2k})^\top_{1\leq k\leq 4}\in\RRR^{16}$.
\end{thm}

\begin{proof}
Consider the equation \eqref{eq:B}.
By the assumption, the matrix $M$ is nonsingular.
By Theorem~\ref{thm:suffcond}, if $\rank_{\RRR}(\underline{T})=4$ then there is $(c_{ik},d_{jk})^\top\in\RRR^{16}$
such that $n_{1k}n_{4k}=n_{2k}n_{3k}$ for $1\leq k\leq 4$.
\end{proof}

Let $M$ be a matrix in \eqref{eq:MatrixM}, and
put $\underline{T}=\underline{X}$  in \eqref{eq:exampletensor}.  
Note that a $4\times4$ matrix
$$(\vect(T_{11}),\vect(T_{12}),\vect(T_{21}),\vect(T_{22}))=\begin{pmatrix}
1&0&0&1\\
0&1&-1&0\\
0&-1&2&0\\
1&0&0&2\\
\end{pmatrix}$$
is nonsingular.  Thus $\rank_{\RRR}(\underline{T})\geq 4$ and $M$ must be nonsingular.
By \eqref{eq:A} we consider the function $f$ from an open subset
$S:=\{ (c_{ik},d_{jk})^\top\mid \det(M)\ne 0\}$ of $\RRR^{16}$ to $\RRR$ defined as \eqref{eq:function}.
Although $f$ might have no minimum value in general,
Theorem~\ref{thm:estimate} yields us that if $\rank_{\RRR}(T)=4$ then $f$ must take the value zero,
and $\inf f>0$ says that $5$ is a typical rank of $\RRR^{2\times2\times2\times2}$.
So we consider the problem: \\
\begin{center}
\begin{tabular}{ll}
infimize  & $f$ \\
subject to & $\det(M)\ne0$ \\
\end{tabular}
\end{center}
We estimate by using the command \lq\lq{}FindMinimum\rq\rq{} in the software Mathematica \cite{Wolfram-Mathematica} and obtains the minimum
value $0.04$ in the $10000$ times iterations.
\begin{table}
\caption{Program in Mathematica}
\begin{verbatim}
T = {{1, 0, 0, 1}, {0, 1, -1, 0}, {0, -1, 1, 0}, {2, 0, 0, 2}};
myf[T_, x11_, x12_, x13_, x14_, x21_, x22_, x23_, x24_, y11_, y12_, 
   y13_, y14_, y21_, y22_, y23_, y24_] := Block[{U, V, mx}, 
   U = {{x11*y11, x11*y21, x21*y11, x21*y21}, 
        {x12*y12, x12*y22, x22*y12, x22*y22}, 
        {x13*y13, x13*y23, x23*y13, x23*y23}, 
        {x14*y14, x14*y24, x24*y14, x24*y24}};
   V = Simplify[T.Inverse[U]]; mx := Sum[V[[i, j]]^2, {i, 1, 4}, {j, 1, 4}];
   Sum[(V[[1, k]]*V[[4, k]] - V[[2, k]]*V[[3, k]])^2, {k, 1, 4}]/mx^2];

tryfind[T_] := Block[{vars, UU, U0, det, a, U, V, f}, a = 0;
  While[a == 0, vars = Table[Random[Real, {-1, 1}], {16}];
   U0 = {{vars[[1]]*vars[[9]], vars[[1]]*vars[[13]], vars[[5]]*vars[[9]], 
      vars[[5]]*vars[[13]]}, {vars[[2]]*vars[[10]], vars[[2]]*vars[[14]], 
      vars[[6]]*vars[[10]], vars[[6]]*vars[[14]]}, {vars[[3]]*vars[[11]], 
      vars[[3]]*vars[[15]], vars[[7]]*vars[[11]], vars[[7]]*vars[[15]]}, 
     {vars[[4]]*vars[[12]], vars[[4]]*vars[[16]], vars[[8]]*vars[[12]], 
      vars[[8]]*vars[[16]]}};
   a = Det[U0]];
  FindMinimum[
   myf[T,x11,x12,x13,x14,x21,x22,x23,x24,y11,y12,y13,y14,y21,y22,y23,y24], 
    {{x11, vars[[1]]}, {x12, vars[[2]]}, {x13, vars[[3]]}, {x14, vars[[4]]}, 
     {x21, vars[[5]]}, {x22, vars[[6]]}, {x23, vars[[7]]}, {x24, vars[[8]]}, 
     {y11, vars[[9]]}, {y12, vars[[10]]},{y13, vars[[11]]},{y14, vars[[12]]}, 
     {y21, vars[[13]]},{y22, vars[[14]]},{y23, vars[[15]]},{y24, vars[[16]]}}]]

res = Table[tryfind[T], {10000}];
val = Table[res[[k, 1]], {k, 1, Length[res]}];
Min[val]
\end{verbatim}
\end{table}
Thus we have

\begin{conj}
The maximal rank of $\RRR^{2\times2\times2\times2}$ is $5$ and
the typical rank of $\RRR^{2\times2\times2\times2}$ is $\{4,5\}$.
\end{conj}

\section{High dimensional tensors}
A lower bound of the maximal rank of $n$-tensors with size $2\times\cdots\times 2$ is
$$\frac{2^n}{2n-n+1}=\frac{2^n}{n+1}$$
(cf. \cite[Proposition~1.2]{Brylinski:2002}) and a canonical upper bound of those is
$2^n$.
We give an upper bound by using the maximal rank of $\FFF^{2\times2\times2\times2}$ tensors.

\begin{prop} \label{prop:highertensorrank}
Let $1\leq s< n$.
$$\mrank_{\KKK}(m_1,m_2,\ldots,m_n)\leq \mrank_{\KKK}(m_1,\ldots,m_s)\prod_{t=s+1}^n m_t$$
\end{prop}

\begin{proof}
Let $\underline{A}=(a_{i_1i_2\ldots i_n})$ be an $n$-tensor with size $m_1\times m_2\times \cdots \times m_n$.

Let $e_{i_1,\ldots,i_k}$, $1\leq i_t\leq a_t, 1\leq t\leq n$ be a standard basis of
$\KKK^{m_1\times\cdots\times m_n}$, that is, 
$\eee_{i_1,\ldots,i_n}$ has $1$ at the $(i_1,\ldots,i_n)$-element
and otherwise $0$.

The tensor $\underline{A}$ is described as
$$\sum_{i_1,\ldots,i_n} a_{i_1,\ldots,i_n}\eee_{i_1,\ldots,i_n} 
=\sum_{i_{s+1},\ldots,i_n} (\sum_{i_1,\ldots,i_s} a_{i_1,\ldots,i_n}\eee_{i_1,\ldots,i_s})\odot\eee_{i_{s+1},\ldots,i_n}.$$

For each $(j_{s+1},\ldots,j_n)$, the $s$-tensor $(a_{i_1\ldots i_sj_{s+1}\ldots j_n})$ is described as 
$$\sum_{k=1}^{\mrank(m_1,\ldots,m_s)} C_{j_{s+1},\ldots,j_n}^{(k)},$$
where $C_{j_{s+1},\ldots,j_n}^{(k)}$ are rank one tensors.
Then
$$\underline{A}
=\sum_{i_{s+1},\ldots,i_n} \sum_{k=1}^{\mrank(m_1,\ldots,m_s)} C_{i_{s+1},\ldots,i_n}^{(k)} \odot\eee_{i_{s+1},\ldots,i_n}$$
which is a sum of $\mrank_{\KKK}(m_1,\ldots,m_s)\prod_{t=s+1}^n m_t$ rank one tensors.
\end{proof}

\begin{cor}
For $n\geq 4$,
The maximal rank of $n$-tensors with size $2\times\cdots\times 2$
over the complex number field
is less than or equal to $2^{n-2}$.
\end{cor}

\begin{proof}
Theorem~\ref{thm:Brylinski} covers the case where $n=4$.
Suppose $n>4$.
The maximal rank of complex $2\times 2\times 2\times 2$ tensors is equal to $4$.
By applying Proposition~\ref{prop:highertensorrank} with $s=4$, we have
$$\mrank_{\CCC}(2,2,\ldots,2)\leq \mrank_{\CCC}(2,2,2,2)\prod_{t=5}^n 2
=4\cdot 2^{n-4}=2^{n-2}.$$
\end{proof}

\begin{lem} \label{lem:essentialrank2}
Let $n$ be a positive integer and let
$A_j$ and $B_j$, $1\leq j\leq n$ be $2\times 2$ real matrices.
There is a rank one real matrix $C$ such that $\rank_{\RRR}(A_j;B_j+C)\leq 2$ 
for any $1\leq j\leq n$.
\end{lem}

\begin{proof}
Put $A_j=\begin{pmatrix} a_j&b_j\\ c_j& d_j\end{pmatrix}$ and
$C=\begin{pmatrix} su&sv\\ tu& tv\end{pmatrix}$.
Since
$$\Delta(A_j;C)=(s(ud_j-vc_j)-t(ub_j-va_j))^2,$$
there exists a rank one matrix $C_0$ such that $\Delta(A_j;C_0)>0$ for any $j\in S_2$.
Let $C=\gamma C_0$.
Since $(A_j;B_j+C)$ is $\{E\}^2\times\GL(2,\RRR)$-equivalent to $(A_j; \gamma^{-1}B_j+C_0)$,
The continuity of $\Delta$ implies that for each $j$, there is $h_j>0$ such that
$\Delta(A_j;B_j+C)>0$ for any $\gamma\geq h_j$  by Proposition~\ref{prop:Silva-etal} (1).
For $C=(\max_j h_j)C_0$, we have $\rank(A_j;B_j+C)\leq 2$ by Proposition~\ref{prop:Silva-etal} (2).
\end{proof}

\begin{thm} \label{thm:higher2x...x2}
Let $k\geq 2$.
The maximal rank of real $k$-tensors with size $2\times\cdots\times 2$
is less than or equal to $2^{k-2}+1$.
\end{thm}

\begin{proof}
The assertion is true for $k=2,3$.
Then suppose that $k\geq 4$.
Let $e_{i_1,\ldots,i_k}$, $i_1,\ldots,i_k=1,2$ be a standard basis of
$(\RRR^{2})^{\otimes k}$, that is, 
$\eee_{i_1,\ldots,i_k}$ has $1$ at the $(i_1,\ldots,i_k)$-element
and otherwise $0$.
Any tensor $\underline{A}$ of $(\RRR^{2})^{\otimes k}$ 
is written by
$$\sum_{i_1,\ldots,i_k} a_{i_1,\ldots,i_k}\eee_{i_1,\ldots,i_k}.$$
This is described as
$$\sum_{i_4,\ldots,i_k} B(i_4,\ldots,i_k)\odot \eee_{i_4,\ldots,i_k},$$
where $B(i_4,\ldots,i_k)=\sum_{i_1,i_2,i_3}
  a_{i_1,\ldots,i_k}\eee_{i_1,i_2,i_3}$ is a $2\times 2\times 2$ tensor.
By Lemma~\ref{lem:essentialrank2}, there is a rank one $2\times2$ matrix $C$ such that
$B(i_4,\ldots,i_k)+(O;C)$ has rank less than or equal to $2$ for any $i_4,\ldots,i_k$.
We have
\begin{equation*}
\begin{split}
\underline{A}&=\sum_{i_4,\ldots,i_k} (B(i_4,\ldots,i_k)+(O;C))\odot \eee_{i_4,\ldots,i_k}
-\sum_{i_4,\ldots,i_k} (O;C)\odot \eee_{i_4,\ldots,i_k} \\
&=\sum_{i_4,\ldots,i_k} (B(i_4,\ldots,i_k)+(O;C))\odot \eee_{i_4,\ldots,i_k}
-C\odot\eee_2\odot \uuu\odot\cdots\odot\uuu,
\end{split}
\end{equation*}
where $\uuu=\begin{pmatrix} 1\\ 1\end{pmatrix}$ and $\eee_2=\begin{pmatrix} 0\\ 1\end{pmatrix}$,
and then
$$
\rank\underline{A}\leq \sum_{i_4,\ldots,i_k} \rank(B(i_4,\ldots,i_k)+(O;C))+1=2^{k-2}+1.
$$
\end{proof}


\end{document}